\documentclass[12pt,leqno]{article}
\usepackage{amsmath,amsthm}
\usepackage{amssymb,latexsym}

\newtheorem{thm}{Theorem}[section]
\newtheorem{cor}{Corollary}[section]
\newtheorem{lem}{Lemma}[section]

\begin{document}

\theoremstyle{definition}

\newtheorem{dfn}{Definition}[section]
\newtheorem{rem}{Remark}[section]
\newtheorem{exa}{Example}[section]

\title{Warped Product Pointwise Semi-slant Submanifolds of Kaehler Manifolds}
\author{Bayram \c{S}ahin}
\date{}
\maketitle

\begin{abstract}
\baselineskip=16pt
It is known that there exist no warped product semi-slant
 submanifolds  in Kaehler manifolds \cite{Sahin}. Recently, Chen and Garay studied pointwise-slant submanifolds of almost Hermitian manifolds in \cite{CG} and obtained many new results for such submanifolds. In this paper, we first introduce pointwise semi-slant submanifolds of Kaehler manifolds and then we show that there exists non-trivial warped product pointwise semi-slant submanifolds of Kaehler manifold by giving an example, contrary to the semi-slant case.  We  present a characterization theorem and establish  an inequality for the squared norm of the second fundamental form in terms of
 the warping function for such warped product submanifolds in Kaehler manifolds. The equality case is also
 considered.
 \end{abstract}

\section{Introduction}

CR-submanifolds of Kaehler manifolds were introduced by Bejancu
\cite{Bejancu} as a generalization of totally real submanifolds
and holomorphic submanifolds.   In \cite{CCR1},  Chen (see
also, \cite{CCR2}, \cite{CCR3}) studied warped product
CR-submanifolds and
 showed that there exist no warped product CR-submanifolds of the
form $M_{\perp} \times_f M_{T}$ such that $M_{\perp}$ is a totally
real submanifold and $M_T$ is a holomorphic submanifold of a Kaehler
manifold $\bar{M}$.  Then he introduced the CR-warped product
submanifolds as follows: A submanifold $M$  of a Kaehler manifold
$\bar{M}$ is called CR-warped product if it is the warped product
$M_T \times_f M_{\perp}$ of a holomorphic submanifold $M_T$ and
a totally real submanifold $M_{\perp}$ of $\bar{M}$.  He also
established general sharp inequalities for CR-warped products in
Kaehler manifolds. After Chen's papers, CR-warped product
submanifolds have been studied by many authors see: a survey \cite{ChenBook} and references therein. \\

On the other hand, slant submanifolds of Kaehler manifolds were
defined by Chen in \cite{CB}  as another generalization of
totally real submanifolds and holomorphic submanifolds. A slant submanifold is
called proper if it is neither totally real nor holomorphic, see also \cite{Chen-Tazawa}  for slant submanifolds. We note
that there exists no inclusion relation between proper
CR-submanifolds and proper slant submanifolds. In \cite{Papa}, N.
Papaghiuc introduced a class of submanifolds, called semi-slant
submanifolds such that the class of CR-submanifolds and the class of
slant submanifolds appear as particular classes of semi-slant
submanifolds. In \cite{Sahin}, we proved that there do not exist
warped product semi-slant submanifolds of the forms $M_{T} \times_f
M_{\theta}$ and $M_{\theta}\times_f M_T$ in Kaehler manifolds, where
$M_{T}$ is a holomorphic submanifold and $M_{\theta}$ is a proper
slant submanifold of a Kaehler manifold $\bar{M}$. Pointwise slant submanifolds of almost Hermitian manifolds were introduced by Etayo in \cite{Etayo} and such submanifolds have been studied by Chen-Garay in \cite{CG}. They obtain simple characterizations, give a method how to construct such submanifolds in Euclidean space  and investigate geometric and topological properties of pointwise slant submanifolds. In this paper we first define pointwise semi-slant submanifolds and then we show that there exists non-trivial warped product pointwise semi-slant submanifolds of the form $M_T\times_f M_{\theta}$ in Kaehler manifolds, where $M_T$ is a holomorphic submanifold and $M_{\theta}$ pointwise slant submanifolds.\\

The paper is organized as follows: In section~2, we present the
basic information needed for this paper. In section~3, we give
definition of pointwise semi-slant submanifolds. After we give two characterization theorems for
pointwise semi-slant submanifolds, we investigate the geometry of leaves of
distributions which are involved in the definition of pointwise semi-slant
submanifolds. In section~4, we prove that there do not exist
warped product submanifolds  of the form $M_{\theta}\times_f
M_{T}$ such that $M_{\theta}$ is a pointwise slant submanifold
and $M_{T}$ is a holomorphic submanifold of $\bar{M}$.  In
section~5, we consider warped product submanifolds of the form
$M_{\theta}\times_f M_{T}$ in Kaehler manifolds, give an example
and present a characterization of such warped product submanifolds.  We
also obtain an inequality for the squared norm of the second
fundamental form in terms of the warping function for warped
product pointwise semi-slant submanifolds. The equality case is also
considered.

In this paper, we assume that every object at hand is
smooth.

\section{Preliminaries}
 \setcounter{equation}{0}
\renewcommand{\theequation}{2.\arabic{equation}}

 Let ($\bar{M},g$) be a Kaehler manifold. This means \cite{YanoKon} that
$\bar{M}$ admits a tensor field $J$ of type (1,1) on $\bar{M}$ such
that, $\forall X,Y \in \Gamma(T\bar{M})$, we have
\begin{equation}
J^2=-I,\quad g(X,Y)=g(JX,JY),\quad(\bar{\nabla}_{X} J)Y=0,
\label{eq:2.1}
\end{equation}
where $g$ is the Riemannian metric and $\bar{\nabla}$ is the Levi-Civita
connection on $\bar{M}$.

Let $\bar{M}$ be a Kaehler manifold with complex structure $J$ and
$M$ a Riemannian manifold isometrically immersed in
$\bar{M}$.  Then $M$ is called holomorphic (complex) if $J (T_p
M)\subset T_p M$, for every $p \in M$, where $T_p M$  denotes the
tangent space of $M$ at the point $p $.  $M$ is called totally real
if $J(T_p M) \subset T_p M^{\perp}$  for every $p \in M,$ where $T_p
M^{\perp}$ denotes the normal space of $M$ at the point $p$. Besides
holomorphic and totally real submanifolds, there are four other
important classes of submanifolds of a Kaehler manifold determined
by the behavior of the tangent bundle of the submanifold under the
action of the complex structure of the ambient manifold.
\begin{enumerate}
    \item [(1)] The submanifold $M$ is called a CR-submanifold
    \cite{Bejancu} if there exists a differentiable distribution
    $D:p~\rightarrow ~D_p \subset T_p M$ such that $D$ is
    invariant with respect to $J$ and the complementary distribution
    $D^{\perp}$ is anti-invariant with respect to $J$.
    \item [(2)] The submanifold $M$ is called slant  \cite{CB} if for
    each
    non-zero vector $X$ tangent to $M$ the angle $\theta (X)$
    between $J X$ and $T_p M$ is a constant, i.e, it does not depend
    on the choice of $p \in M$ and $X \in T_p M$.
    \item [(3)] The submanifold $M$ is called semi-slant
    \cite{Papa} if it is endowed with two orthogonal distributions
    $D$ and $D',$ where $D$ is invariant with respect to $J$ and
    $D'$ is slant, i.e, $\theta (X) $ between $J X$ and $D'_p$ is
    constant for $X \in D'_p$.
    \item [(4)] The submanifold $M$ is called pointwise slant submanifold \cite{Etayo}, \cite{CG} if at each given point $p\in M$, the Wirtinger angle $\theta(X)$ between  $JX$ and the space $T_pM$ is independent of the choice of the nonzero vector $X \in \Gamma(TM)$. In this case, the angle $\theta$ can be regarded as a function $M$, which is called the {\it  slant function} of the pointwise slant submanifold.

\end{enumerate}
A point $p$ in a pointwise slant submanifold is called a totally real point if its slant function $\theta$
satisfies $\cos \theta = 0$ at $p$. Similarly, a point $p$ is called a complex point if its slant function satisfies $\sin \theta = 0$ at $p$.
A pointwise slant submanifold $M$ in an almost Hermitian manifold $\bar{M}$ is called totally real if every point
of $M$ is a totally real point. A pointwise slant submanifold of an almost Hermitian manifold is called pointwise proper slant
if it contains no totally real points. A pointwise slant submanifold $M$ is called slant when its slant function $\theta$ is globally
constant, i.e., $\theta$ is also independent of the choice of the point on $M$. It is clear that pointwise slant submanifolds include holomorphic and totally real submanifolds and slant submanifolds.  It is also clear that CR-submanifolds
and slant submanifolds  are particular semi-slant submanifolds with
$\theta = \frac{\pi}{2}$ and $D= \{ 0 \}$, respectively.\\

Let $M$ be a Riemannian manifold isometrically immersed in
$\bar{M}$ and denote by the same symbol $g$ for the Riemannian
metric induced on $M$. Let $\Gamma(TM)$ be the Lie algebra of
vector fields in $M$ and $\Gamma(TM^{\perp})$ the set of all
vector fields normal to $M$, same notation for smooth sections of
any other vector bundle $E$.
 Denote by $\nabla$ the Levi-Civita connection of $M$. Then the
 Gauss and Weingarten formulas are given  by
 \begin{equation}
 \bar{\nabla}_X Y=\nabla_X Y+h(X,Y) \label{eq:2.2}
 \end{equation}
 and
 \begin{equation}
 \bar{\nabla}_X N=-A_N X+{\nabla}^{\perp}_ X N \label{eq:2.3}
 \end{equation}
 for any $X, Y \in \Gamma(TM)$ and any $N \in \Gamma(TM^{\perp}),$ where
 $\nabla^{\perp}$ is the connection in the normal bundle $TM^{\perp}$, $h$ is
 the second fundamental form of $M$ and $A_N$ is the Weingarten
 endomorphism associated with $N$.  The second fundamental form $h $ and the shape operator $A$ are related by
  \begin{equation}
  g(A_N X, Y)=g(h(X,Y), N). \label{eq:2.4}
  \end{equation}

 For any $X \in \Gamma(TM)$ we write
 \begin{equation}
 JX=TX+F X, \label{eq:2.5}
 \end{equation}
where $TX$ is the tangential component of $JX$ and $F X$ is the
normal component of $JX$.  Similarly, for any vector field $N$ normal to
$M$, we put
\begin{equation}
JN=BN+CN, \label{eq:2.6}
\end{equation}
where $BN$ and $CN$ are the tangential and the normal components of
$JN$, respectively.\\

\section{Pointwise Semi-slant Submanifolds}
\setcounter{equation}{0}
\renewcommand{\theequation}{3.\arabic{equation}}
In this section, we define and study pointwise semi-slant submanifolds
in a Kaehler manifold $\bar{M}$.  We  obtain characterizations, give an example and investigate the geometry of leaves of distributions.\\

\begin{dfn} Let $\bar{M}$ be a Kaehler manifold
and $M$ a real submanifold of $\bar{M}$.  Then  we say that $M$ is
a pointwise semi-slant submanifold if there exist two orthogonal
distributions $D^{T}$ and $D^{\theta}$ on $M$ such that
\begin{enumerate}
    \item [(a)] $TM$ admits the orthogonal direct decomposition $TM=
    D^T \oplus D^{\theta}$.
    \item [(b)] The distribution $D^{T}$ is a holomorphic distribution, i.e.,
    $J D^T =D^T$.
    \item [(c)]  The distribution $D^{\theta}$ is pointwise slant with slant function  $\theta $.
\end{enumerate}
\end{dfn}

In this case, we call the angle $\theta $ the slant function of the
pointwise slant submanifold $M$.  The holomorphic distribution
$D^T$ of a pointwise semi-slant submanifold  is a pointwise slant distribution
with slant function $\theta =0$. If we denote the
dimension of $D^T$ and $D^{\theta}$  by $m_1$ and $m_2$,
respectively, then we have the
following:
\begin{enumerate}
    \item [(a)] If $m_2=0$, then $M$ is a holomorphic submanifold.
    \item [(b)] If $m_1 =0$,  then $M$ is a pointwise slant submanifold.
    \item [(c)] If  $\theta$ is constant then $M$ is a proper semi-slant submanifold with slant angle $\theta
    $.
    \item [(d)] If $\theta =\frac{\pi }{2},$ then $M$ is a CR-submanifold.
\end{enumerate}

We say that a pointwise semi-slant submanifold is proper if
    $m_1 \neq 0$ and $\theta$ is not a constant.\\

\begin{exa} Let $M$ be a submanifold of
$\mathbf{R}^6$ given by
$$\chi(t,s,u,v)=(t,s,u,\sin\,v, 0,\cos\, v).$$
It is easy to see that a local frame of $TM$ is given by
$$
Z_1 =\frac{\partial}{\partial x_1},Z_2=\frac{\partial}{\partial x_2},Z_3 =\frac{\partial}{\partial x_3},Z_4=\cos\,v \frac{\partial}{\partial x_4}-\sin\,v\frac{\partial}{\partial
x_6}.$$
Then using the canonical complex structure of $\mathbf{R}^6$, we
see that $D^{T}=span\{ Z_1,Z_2\}$. Moreover it is easy to see that $D^{\theta} =span \{Z_3, Z_4 \}$
is a pointwise slant distribution with slant function $v$. Thus $M$ is
a proper pointwise semi-slant submanifold of $\mathbf{R}^6$. \\
\end{exa}

Let $M$ be a pointwise semi-slant submanifold of a Kaehler manifold
$\bar{M}$.  We denote the projections on the distributions $D^{T} $
and $D^{\theta}$ by $P_1 $ and $P_2 $, respectively. Then we can
write
\begin{equation}
X= P_1 X + P_2 X \label{eq:3.1}
\end{equation}
for any $X \in \Gamma(TM)$.  Applying $J$ to (\ref{eq:3.1}) and using
(\ref{eq:2.5})  we obtain
\begin{equation}
JX =JP_1 X +TP_2 X +FP_2 X. \label{eq:3.2}
\end{equation}
Thus we have
\begin{eqnarray}
JP_1 X \in \Gamma(D^T) &,& FP_1 X=0, \label{eq:3.3}\\
TP_2 X \in \Gamma(D^{\theta}) &,& FP_2 X \in \Gamma(TM^{\perp}).
\label{eq:3.4}
\end{eqnarray}
Then (\ref{eq:3.3}) and (\ref{eq:3.4}) imply
\begin{equation}
TX = JP_1 X+TP_2 X \label{eq:3.5}
\end{equation}
for $X \in \Gamma(TM)$. \\

It is known that  $M$ is a pointwise slant submanifold of $\bar{M}$ if and only
if
\begin{equation}
T^2 = -({\cos}^2\,\theta) I\label{eq:3.6}
\end{equation}
for some real-valued function $\theta$ defined on $M$  \cite{CG}, where $I$
denotes the identity transformation of the tangent bundle $TM$ of
the submanifold $M$.    Thus we can prove the following characterization theorem.\\

\begin{thm} Let $D$ be a distribution on $M$.
Then $D$ is pointwise slant if and only if there exists a function
$\lambda \in [-1,0]$ such that $(TP_2)^2 X = \lambda\, X $ for
$X \in \Gamma(D),$ where $P_2$ denotes the orthogonal projection on
$D$. Moreover in this case $\lambda = -\cos^2\theta $.\end{thm}

Actually this theorem is similar to that theorem given  \cite{CCFF}
for Sasakian case. We can use Theorem~3.1 to characterize pointwise
semi-slant submanifolds. Let $M$ be a real submanifold of an almost
Hermitian manifold $\bar{M}$ and $D$ a distribution on $M$. We
define $T_{D}:D\longrightarrow TM$ by
$T_{D}(X)=(JX)^{\mathcal{T}_{D}}$, where $\mathcal{T}_{D}$ is the
orthogonal projection of $T\bar{M}$ onto $D$. If $M$ is a pointwise
slant submanifold and $D$ is its slant distribution, we have
$$T_{D}=P_2TI_{D},$$
where $I_{D}$ is the identity of $D$.\\

\begin{thm} Let $M$ be a submanifold of a
Kaehler manifold $\bar{M}$.  Then $M$ is a pointwise semi-slant
submanifold if and only if there exists a function $\lambda \in
[-1,0]$ and a distribution $D$ on $M$ such that
\begin{enumerate}
    \item [(i)] $D=\{ X \in \Gamma(TM) \mid (T_{D})^2 X = \lambda X
    \}$,
    \item [(ii)] {\sl $T$ maps $D$ into $D$. }
\end{enumerate}
Moreover in this case $\lambda = -\cos^2\theta $, where $\theta $
denotes the slant function of $M$.\end{thm}

\begin{proof} Let $M$ be a pointwise semi-slant submanifold
of $\bar{M}$.  Then $\lambda = -\cos^2 \theta$ and $D=D^{\theta} $.
By the definition of pointwise semi-slant submanifold, (ii) is
clear. Conversely (i) and (ii) imply $TM=D \oplus D^T$.  Since $T$
maps $D$ into $D$, it implies that  $J(D^T)=D^T$. Thus proof is
complete.\end{proof} From Theorem~3.2 we have the following
corollary:
\begin{cor} Let $M$ be a pointwise semi-slant
submanifold of a Kaehler manifold $\bar{M}$.  Then we have
\begin{eqnarray}
g(TX,TY)&=& \cos^2\theta \, g(X,Y) \label{eq:3.10} \\
g(FX,FY)&=& \sin^2 \theta \, g(X,Y) \label{eq:3.11}
\end{eqnarray}
for $X, Y \in \Gamma(D^{\theta})$.
\end{cor}

\begin{proof} For $X, Y \in \Gamma(D^{\theta})$, from
(\ref{eq:2.1}) we have $g(TX,TY)=g(JX-FX,TY)$.
Hence$g(TX,TY)=-g(X,JTY)$.  Using Theorem~3.2 (i), we obtain
(\ref{eq:3.10}). Using (\ref{eq:3.10}) we get
(\ref{eq:3.11}).\end{proof}

In the rest of this section, we first study  the integrability of
distributions  and then we find  the conditions under which leaves of
distributions on a pointwise semi-slant submanifold $M$ in a Kaehler
manifold $\bar{M}$ are totally geodesic immersed in $M$.  For the integrability of the distributions $D^{T}$  and $D^{\theta}$ on a pointwise semi-slant
submanifold $M$, we have the following.

\begin{thm}
Let $M$ be a proper pointwise semi-slant submanifold of a Kaehler manifold.
\begin{enumerate}
  \item [(i)] The distribution $D^T$ is integrable if and only if
$$g(h(X,JY),FV)=g(h(JX,Y),FV),\, \forall X,Y\in \Gamma(D^{T})\quad \mathrm{\it and}\quad V\in \Gamma(D^{\theta}).$$
  \item [(ii)] The distribution $D^{\theta}$ is integrable if and only if
$$g(A_{FTW}V-A_{FTV}W,X)=g(A_{FW}V-A_{FV}W,JX)$$
for $W \in \Gamma(D^{\theta})$.
\end{enumerate}
\end{thm}
\begin{proof}We prove (i), (ii) can be obtained in a similar way. From (\ref{eq:2.1}), (\ref{eq:2.3}) and (\ref{eq:2.5}) we have
\begin{eqnarray}
g([X,Y],V)&=&-g(\bar{\nabla}_XY,T^2V+FTV)+g(h(X,JY),FV)\nonumber\\
&&+g(\bar{\nabla}_YX,T^2V+FTV)-g(h(JX,Y),FV).\nonumber
\end{eqnarray}
Then the symmetric $h$ and (\ref{eq:3.6}) imply that
$$
\sin^2\theta g([X,Y],V)=g(h(X,JY),FV)-g(h(JX,Y),FV)
$$
which gives the assertion.
\end{proof}

Next we give necessary and sufficient conditions for the distributions $D^T$ and $D^{\theta}$  whose leaves are totally geodesic.
\begin{thm}Let $M$ be a proper pointwise semi-slant submanifold of a Kaehler manifold.
\begin{enumerate}
\item [(a)] The holomorphic distribution $D^T$ defines a totally geodesic foliation if and only if
\begin{equation}
g(h(X,Y),FTV)=g(h(X,JY),FV)\label{eq:}
\end{equation}
for $X,Y\in \Gamma(D^T)$ and $V\in \Gamma(D^{\theta})$.
\item [(b)]The
slant distribution $D^{\theta}$  defines a totally geodesic foliation on $M$  if and only if
$$g(h(U,X),FTV)=g(h(U,JX),FV)$$
 for $X\in \Gamma(D^T)$ and $U,V\in \Gamma(D^{\theta})$.
 \end{enumerate}
 \end{thm}

\begin{proof} Let $M$ be a   proper pointwise semi-slant submanifold of a Kaehler manifold $\bar{M}$.  Then  we have
$g(\nabla_XY,V)=g(\bar{\nabla}_XJY,JV)$ for $X,Y\in \Gamma(D^T)$ and $V\in \Gamma(D^{\theta})$. Thus using (\ref{eq:2.5}) and (\ref{eq:2.6}) we get
$$g(\nabla_XY,V)=-g(\bar{\nabla}_XY,JTV)+g(\bar{\nabla}_XJY,FV).$$ Then (\ref{eq:3.6}) implies that
$$\sin^2\,\theta g(\nabla_XY,V)=-g(h(X,Y),FTV)+g(h(X,JY),FV)$$
which gives (a). In a similar way, we obtain (b).\end{proof}

 Thus from Theorem~3.4, we have the following
 result:

 \begin{cor} Let $M$ be a pointwise semi-slant
 submanifold of a Kaehler manifold $\bar{M}$.  Then $M$ is a locally
 Riemannian product manifold $M=M_{T} \times M_{\theta}$ if and
 only if
 $$A_{FTV} X=A_{FV} JX$$
  for $V \in \Gamma(D^{\theta})$ and $X \in \Gamma(D^{\perp})$,
where $M_{T}$ is a holomorphic submanifold and $M_{\theta}$ is
a pointwise slant submanifold of $\bar{M}$. \end{cor}

\section{Warped Products $M_{\theta} \times_f M_{T}$ in Kaehler Manifolds}

\setcounter{equation}{0}
\renewcommand{\theequation}{4.\arabic{equation}}  Let $(B, g_1)$ and $(F, g_2) $ be two
Riemannian manifolds, $f:B \rightarrow (0,\infty) $ and $\pi: B
\times F \rightarrow B$, $\eta:B \times F \rightarrow F$ the
projection maps given by $\pi (p,q)=p$ and $\eta (p,q)=q$ for every
$(p,q) \in B \times F$.  The warped product (\cite{BO})
$M=B\times_{f} F$ is the manifold $B \times F$ equipped with the
Riemannian structure such that
$$ g(X,Y)=g_1(\pi_* X,\pi_* Y)+(f o \pi)^2
g_2(\eta_* X,\eta_* Y)$$ for every $X$ and $Y$ of $M$, where $*$
denotes the tangent map. The function $f$ is called the warping
function of the warped product manifold. In particular, if the
warping function is constant, then the warped product manifold $M$ is said to be
trivial. \\

 Let $ X, Y$ be vector fields on $B$ and $V,W$ vector
fields on $F$, then from Lemma~7.3 of \cite{BO}, we have
\begin{equation}
\nabla_X V=\nabla_V X=(\frac{Xf}{f})V \label{eq:4.1}
\end{equation}
where $\nabla$ is the Levi-Civita connection on $M$.\\

In this section we investigate  the existence of non-trivial warped
product submanifolds $M_{\theta} \times_f M_{T}$ of Kaehler
manifolds such that $M_{T}$ is a holomorphic submanifold and
$M_{\theta}$ is a proper pointwise slant submanifold of $\bar{M}$.

 \begin{thm} Let $\bar{M}$ be a Kaehler
 manifold. Then there exist no non-trivial warped product submanifolds $M=M_{\theta} \times_f M_{T}$ of a Kaehler manifold $\bar{M}$ such that
 $M_{T}$ is a holomorphic submanifold and $M_{\theta}$ is a
 proper pointwise slant submanifold of $\bar{M}$. \end{thm}

 \begin{proof} From (\ref{eq:4.1}), (\ref{eq:2.2}), (\ref{eq:2.3}), (\ref{eq:2.4}) and (\ref{eq:2.5}) we
 have
$$
 V(lnf)g(X,Y)=-g(\bar{\nabla}_XT^2V+FTV,Y)-g(A_{FV}X,JY).$$
Using (\ref{eq:3.6}) we get
$$
 V(lnf)g(X,Y)=g(\bar{\nabla}_X\cos^2\,\theta V,Y)-g(\bar{\nabla}_XFTV,Y)-g(A_{FV}X,JY).$$
Thus from (\ref{eq:2.3}) and (\ref{eq:2.4}) we obtain
\begin{eqnarray}
 V(lnf)g(X,Y)&=&-\sin\,2\theta X(\theta)g(V,Y)+\cos^2\,\theta g(\nabla_XV,Y)\nonumber\\
 &&+g(h(X,Y),FTV)-g(h(X,JY),FV).\nonumber
\end{eqnarray}
Since $D^T$ and $D^{\theta}$ are orthogonal, using (\ref{eq:4.1}) we arrive at
 $$
 \sin^2\,\theta V(lnf)g(X,Y)=g(h(X,Y),FTV)-g(h(X,JY),FV).$$
  Interchanging the role of $X$ and $Y$ in above equation and then subtracting each other,  we derive
 \begin{equation}
 g(h(JX,Y),FV)=g(h(X,JY),FV).\label{eq:4.2}
 \end{equation}
 On the other hand, from (\ref{eq:2.3}), (\ref{eq:2.1}), (\ref{eq:2.5}) and (\ref{eq:4.1}) we have
 \begin{equation}
 g(h(X,JY),FV)=-V(lnf)g(X,Y)+TV(lnf)g(X,JY).\label{eq:4.3}
 \end{equation}
 Then from (\ref{eq:4.2}) and (\ref{eq:4.3}) we conclude
 $$TV(lnf)g(X,JY)=0.$$
Replacing $V$ by $TV$ and $X$ by $JX$ we find
$$\cos^2\,\theta V(lnf)g(X,Y)=0$$
which implies
$$V(lnf)=0$$
due to $M_{\theta}$ is proper pointwise slant and $M_T$ is a Riemannian manifold. Thus it follows that $f$ is a constant.

\end{proof}

\begin{rem} We note that Theorem~4.1 is a
generalization of Theorem~3.1 in \cite{CCR1} and Theorem 3.1 in \cite{Sahin}.\end{rem}

\section{ Non-trivial Warped Products $ M_T  \times_f M_{\theta}$ in Kaehler Manifolds}
\setcounter{equation}{0}
\renewcommand{\theequation}{5.\arabic{equation}} Theorem~4.1 shows that there do not exist non-trivial warped product pointwise semi-slant submanifolds of the form $M_{\theta} \times_f
M_{T}$ in Kaehler manifolds. In this section,  we consider
non-trivial warped product pointwise semi-slant submanifolds of the
form $M_{T}\times_f M_{\theta}$, where $M_{T}$ is a holomorphic
submanifold  and $M_{\theta}$ is a proper pointwise slant
submanifold of $\bar{M}$.  First, we are going to give an example of
non-trivial warped product pointwise semi-slant submanifold of the form $ M_T \times _f M_{\theta}$. \\

\begin{exa}For $t,s\neq 0,1, u,v\in (0,\frac{\pi}{2})$, consider a submanifold $M$ in $R^{10}$
given by the equations
\begin{eqnarray}
x_1=t\cos\,u, x_2=s\cos\, u&,& x_3=t\cos\, v, x_4= s\cos\,v, x_5=t\sin\,u\nonumber\\
x_6=s\sin\,u, x_7=t\sin\,v&,&  x_8=s\sin\,v, x_9=u, x_{10}=v.\nonumber
\end{eqnarray}

Then the tangent bundle $TM$ is spanned by $Z_1, Z_2, Z_3$ and $Z_4$
where
\begin{eqnarray}
Z_1&=& \cos u\,\frac{\partial}{\partial x_1} +\cos v \,
\frac{\partial}{\partial x_3}+\sin u\, \frac{\partial}{\partial x_5}+\sin v\, \frac{\partial}{\partial x_7}\nonumber\\
Z_2&=& \cos u\,\frac{\partial}{\partial x_2} +\cos v \,
\frac{\partial}{\partial x_4}+\sin u\, \frac{\partial}{\partial x_6}+\sin v\, \frac{\partial}{\partial x_8}\nonumber\\
Z_3&=& -t \, \sin u \, \frac{\partial}{\partial
x_1}-s\,\sin u\,\frac{\partial}{\partial x_2} +t\,\cos u
\frac{\partial}{\partial x_5}+ s\,\cos
u\,\frac{\partial}{\partial x_6}+\frac{\partial}{\partial x_9} \nonumber\\
Z_4&=& -t \, \sin v \, \frac{\partial}{\partial
x_3}-s\,\sin v\,\frac{\partial}{\partial x_4} +t\,\cos v
\frac{\partial}{\partial x_7}+ s\,\cos
v\,\frac{\partial}{\partial x_8}+\frac{\partial}{\partial x_{10}}. \nonumber
\end{eqnarray}
Then $D^T=span \{Z_1,Z_2 \}$ is a holomorphic distribution
and $D^{\theta}=span\{Z_3, Z_4 \}$ is a pointwise slant distribution with
the slant function ${\cos^{-1} (\frac{1}{t^2+s^2+1}})$.  Thus
$M$ is a pointwise semi-slant submanifold of $R^{10}$.  It is easy to see that
$D^{\theta}$ and $D^T$ are integrable. We denote the integral manifolds of
$D^{T}$ and $D^{\theta}$ by $M_{T}$ and $M_{\theta}$,
respectively. Then  the metric tensor $g$ of $M$ is
\begin{eqnarray}
g&=&2dx^2_1+2d^2_2+(t^2+s^2+1)(dx^2_3+dx^2_4)\nonumber\\
&=&g_{M_{T}}+(t^2+s^2+1)g_{M_{\theta}}. \nonumber
\end{eqnarray}
Thus $M$ is a non-trivial warped product submanifold of $R^{10}$ of
the form $M_T\times_f M_{\theta}$ with warping function
$\sqrt{(t^2+s^2+1)}$. \end{exa}

\begin{rem}Non-trivial warped product pointwise semi-slant submanifolds of the form $M_T\times_fM_{\theta}$ are natural extension of warped product CR-submanifolds. Indeed, every CR-warped
product submanifold is a non-trivial warped product pointwise
semi-slant submanifold of the form $M_T \times_f M_{\theta}$ with
the slant function $\theta=0$.
\end{rem}

From now on, we will consider non-trivial warped product pointwise
semi-slant submanifold $M=M_T \times_f M_{\theta}$ such that
$M_{\theta}$ is a proper pointwise slant submanifold and $M_T$ is a
holomorphic submanifold. First we give some preparatory lemmas.

\begin{lem} Let $M=M_T \times_f
M_{\theta}$ be a non-trivial warped product pointwise proper
semi-slant submanifold of a Kaehler manifold $\bar{M}$. Then we have
\begin{equation}
g(A_{FV}W,X)=g(A_{FW}V,X)\label{eq:5.1}
\end{equation}
for $V,W\in \Gamma(D^{\theta})$ and $X\in \Gamma(D^T)$.
\end{lem}
\begin{proof}
Using (\ref{eq:2.1}), (\ref{eq:2.2}) and (\ref{eq:2.5}) we have
$$g(A_{FV}X,W)=g(\nabla_XV,TW)+g(\bar{\nabla}_XV,FW )+g(\nabla_XTV,W)$$
for $X\in \Gamma(D^T)$ and $V,W \in \Gamma(D^{\theta})$. Then from (\ref{eq:4.1}) and (\ref{eq:2.2}) we obtain
$$g(A_{FV}X,W)=g(h(X,V),FW)$$
which gives the assertion.
\end{proof}
\begin{lem} Let $M=M_T \times_f
M_{\theta}$ be a non-trivial warped product pointwise semi-slant
submanifold of a Kaehler manifold $\bar{M}$. Then we have
\begin{equation}
g(A_{FTW}V,X)=-JX(lnf)g(TW,V)-X(lnf)\cos^2\theta g(V,W)\label{eq:5.2}
\end{equation}
and
\begin{equation}
g(A_{FW}V,JX)=X(lnf)g(W,V)+JX(lnf) g(V,TW)\label{eq:5.3}
\end{equation}
for $V,W\in \Gamma(D^{\theta})$ and $X\in \Gamma(D^T)$.
\end{lem}
\begin{proof}
From (\ref{eq:5.1}) we write $g(A_{FTW}V,X)=g(A_{FV}TW,X)$. Then using (\ref{eq:2.1}), (\ref{eq:2.2}), (\ref{eq:2.3}) and (\ref{eq:2.5}) we have
$$g(A_{FTW}X,W)=g(\nabla_{TW}V,JX)+g(\nabla_{TW}TV,X).$$
 Thus from (\ref{eq:4.1}) and (\ref{eq:3.10}) we obtain (\ref{eq:5.2}). (\ref{eq:5.2}) gives (\ref{eq:5.3}).
\end{proof}

In the sequel we give a characterization for non-trivial warped
product pointwise semi-slant submanifolds of the form  $M_{T}
\times_f M_{\theta}$.   Recall that we have the following result of
Hiepko \cite{Hiepko}, see also\cite{Dillen}:
 {\sl Let $D_1$ be a vector
subbundle in the tangent bundle of a Riemannian manifold $M$ and
$D_2$ be its normal bundle. Suppose that the two distributions are
involutive. We denote the integral manifolds of $D_1$ and $D_2$ by
$M_1$ and $M_2$, respectively. Then $M$ is locally isometric to
non-trivial warped product $M_1 \times_f M_2$ if the integral
manifold $M_1$ is totally geodesic and the integral manifold $M_2$
is an extrinsic sphere, i.e, $M_2$  is a totally
umbilical submanifold with parallel mean curvature vector.}

\begin{thm}Let $M$ be a pointwise semi-slant submanifold of a Kaehler manifold $\bar{M}$. Then $M$ is locally a non-trivial warped product manifold of the form $M=M_T \times_f M_{\theta}$ such that
$M_{\theta}$ is a proper pointwise slant submanifold and $M_T$ is a
holomorphic submanifold in $\bar{M}$ if the following condition is satisfied
\begin{equation}
A_{FTW}X-A_{FW}JX=-(1+\cos^2\,\theta)X(\mu)W\label{eq:5.4}
\end{equation}
where $\mu$ is a function on $M$ such that $W(\mu) = 0$ for every $W\in \Gamma(D^{\theta})$ and $X\in \Gamma(D^T)$.
\end{thm}
\begin{proof} Let $M=M_T \times_f
M_{\theta}$ be a non-trivial warped product pointwise semi-slant
submanifold of a Kaehler manifold $\bar{M}$. Then from
(\ref{eq:2.1}), (\ref{eq:2.3}) and (\ref{eq:2.5}) we obtain
$$g(A_{FV}X,Y)=g(\nabla_XV,JY)+g(\nabla_XTV,Y)$$
for $X,Y\in \Gamma(D^T)$ and $V\in \Gamma(D^{\theta})$. Then using (\ref{eq:4.1}) we derive
$$g(A_{FV}X,Y)=0$$
which shows that $A_{FV}X$ belongs to $D^{\theta}$. Conversely, suppose that $M$ is a pointwise semi-slant submanifold of a Kaehler manifold $\bar{M}$ such that
\begin{equation}
A_{FTW}X-A_{FW}JX=-(1+\cos^2\,\theta)X(\mu)W\label{eq:500.1}
\end{equation} for
$W\in \Gamma(D^{\theta})$ and $X\in \Gamma(D^T)$. Then from Theorem
3.3 (ii), $D^{\theta}$ is integrable.  Also from Theorem 3.4 (b), we find
that the integral manifold $M_T$ of $D^T$ is totally geodesic. Let
$M_{\theta}$ be the integral manifold of $D^{\theta}$ and denote the
second fundamental form of $M_{\theta}$ in $M$ by $h_{\theta}$.
Since Weingarten operator $A_N$ is self-adjoint, using
(\ref{eq:2.3}) we get
$$g(A_{FTV}X-A_{FV}JX,W)=-g(X,\bar{\nabla}_WFTV)+g(JX,\bar{\nabla}_WFV)$$
for $V, W\in \Gamma(D^{\theta})$ and $ X\in \Gamma(TM)$.
Then from (\ref{eq:2.1}), (\ref{eq:2.2}) and (\ref{eq:2.5}) we have
$$g(A_{FTV}X-A_{FV}JX,W)=g(X,\nabla_WT^2V)+g(X,\nabla_WV).$$
Thus from (\ref{eq:3.6}) we obtain
\begin{eqnarray}
g(A_{FTV}X-A_{FV}JX,W)&=&\sin\,2\theta W(\theta)g(X,V)-\cos^2\,\theta g(X,\nabla_WV)\nonumber\\
&&+g(X,\nabla_WV)\nonumber\\
&=&\sin^2\,\theta g(X,\nabla_WV).\nonumber
\end{eqnarray}
Hence we derive
\begin{equation}
g(A_{FTV}X-A_{FV}JX,W)=\sin^2\,\theta\,g(X,h_{\theta}(V,W)).\label{eq:5.5}
\end{equation}
Then (\ref{eq:5.5}) and (\ref{eq:500.1}) imply that
$$h_{\theta}(V,W)=-(\csc^2\,\theta+\cot^2\,\theta)\nabla\mu g(V,W)$$
which shows that $M_{\theta}$ is a totally umbilical submanifold in $M$ with the mean curvature vector field $-(\csc^2\,\theta+\cot^2\,\theta)\nabla\mu$, where $\nabla \mu$ is the gradient of $\mu$.  On the other hand,
by direct computations, we get
\begin{eqnarray}
g(\nabla_V  \nabla \mu,X)&=&[Vg(\nabla \mu, X)-g(\nabla \mu,
\nabla_V
X)]\nonumber\\
&=&[V(X(\mu))-[V,X]\mu-g(\nabla \mu, \nabla_X
V)]\nonumber\\
&=&[[V,X]\mu+X(V(\mu))-[V,X]\mu-g(\nabla \mu,\nabla_X
V)]\nonumber\\
&=&[X(V(\mu))-g(\nabla \mu,\nabla_X V)].\nonumber
\end{eqnarray}
Since $V(\mu)=0$, we obtain
$$g(\nabla_V  \nabla \mu,X)=g(\nabla \mu,\nabla_X V).$$ On the other hand, since  $\nabla \mu
\in \Gamma(TM_{T})$ and $M_{T}$ is totally geodesic in
$M,$ it follows that $\nabla_X V \in \Gamma(TM_{\theta})$ for $X \in
\Gamma(D^{T})$ and $V \in \Gamma(D^{\theta}) $.  Hence $g(\nabla_V
 \nabla \mu,X)=0$.  Then the spherical condition is also
fulfilled, that is $M_{\theta}$ is an extrinsic sphere in $M$.
 Thus we conclude that $M$ is a non-trivial warped product and proof is
 complete.\end{proof}
We now give an inequality in terms of the length of the second fundamental form. First we give  a lemma which will be useful for the theorem.
\begin{lem}Let $M=M_T \times_f
M_{\theta}$ be a non-trivial warped product pointwise semi-slant
submanifold of a Kaehler manifold $\bar{M}$. Then we have
\begin{equation}
g(h(X,Y),FV)=0\label{eq:5.6}
\end{equation}
and
\begin{equation}
g(h(X,V),FW)=-JX(lnf)g(V,W)-X(lnf)g(V,TW)\label{eq:5.7}
\end{equation}
for $V,W\in \Gamma(D^{\theta})$ and $X,Y \in \Gamma(D^T)$.
\end{lem}
\begin{proof}
From (\ref{eq:2.5}), (\ref{eq:2.1}) and (\ref{eq:2.2}) we get
$$g(h(X,Y),FV)=-g(\nabla_XJY,V)-g(\nabla_XY,TV).$$
Since $D^T$ and $D^{\theta}$ are orthogonal, using (\ref{eq:4.1}) we derive
$$g(h(X,Y),FV)=X(lnf)g(V,JY)+X(lnf)g(TV,Y)=0$$
which gives (\ref{eq:5.6}). (\ref{eq:5.7}) comes from (\ref{eq:5.2}) and (\ref{eq:5.3}).
\end{proof}

Let $M$ be an $(m+n)$ dimensional proper pointwise semi-slant
submanifold of a Kaehler manifold $\bar{M}^{m+2n}$, where $\bar{M}$
is of real dimension $m+2n$ and it is obvious that $m$ is also even.
Then we choose a canonical orthonormal frame $\{e_1,...,e_m,
\bar{e}_1,...,\bar{e}_n, e^*_{1},...,e^*_{n}\}$ of $\bar{M}$ such
that, restricted to $M$, $e_1,...,e_m, \bar{e}_1,...,\bar{e}_n$ are
tangent to $M$. Then $\{e_1,...,e_m, \bar{e}_1,...,\bar{e}_n\}$ form
an orhonormal frame of $M$. We can take $\{e_1,...,e_m,
\bar{e}_1,...,\bar{e}_n\}$ in such a way that   $\{ e_1,...,e_m \}$
form an orthonormal frame of $D^T$ and $\{
\bar{e}_{1},..,\bar{e}_{n}\}$ form an orhonormal frame of
$D^{\theta}$, where $dim(D^T)=m$ and $dim(D^{\theta})=n$. We can
take $\{e^*_{1},...,e^*_{n}\}$ as an orthonormal frame of
$F(D^{\theta})$. It is known that a proper pointwise slant
submanifold is always even dimensional. Hence, $n=2p$. Then we can
choose orthonormal frames  $\{ \bar{e}_1,..,\bar{e}_{2p}\}$ of
$D^{\theta}$ and $\{ e^*_{1},...,e^*_{n}\}$ of  $F(D^{\theta})$ in
such a way that
\begin{eqnarray}
\bar{e}_{2}=\sec  \theta\, T\bar{e}_{1},.&.&.,\bar{e}_{2p}=\sec
\theta\, T\bar{e}_{2p-1}
\nonumber\\
e^*_{1}=\csc \theta F\bar{e}_{1},.&.&.,e^*_{2p}=\csc \theta\,
F\bar{e}_{2p}, \nonumber
\end{eqnarray}
where $\theta$ is the slant function. We will call this orthonormal frame an adapted frame as for slant submanifold case \cite{CB}.

\begin{thm} Let $M$ be an ($m+n$)-dimensional
non-trivial warped product pointwise semi-slant  submanifold of the form $M_T\times_f
M_{\theta}$ in a Kaehler manifold $\bar{M}^{m+2n}$, where $M_{T}$ is
a holomorphic submanifold and $M_{\theta}$ is a proper pointwise
slant submanifold of $\bar{M}^{m+2n}$. Then we have
\begin{enumerate}
  \item [(i)]{\sl The squared norm of the second fundamental form of M
  satisfies}
  \begin{equation}
  \parallel h \parallel^2 \geq 2n\, (\csc^2\,\theta+\cot^2\theta)\, \parallel \nabla (\ln f)
  \parallel^2, \quad dim(M_{\theta})=n.\label{eq:5.8}
  \end{equation}
  \item [(ii)] {\sl If the equality  of (\ref{eq:5.8})  holds identically, then $M_T$  is a totally geodesic submanifold
and $M_{\theta}$  is a totally umbilical submanifold of $\bar{M}$.
Moreover, $M$ is a minimal submanifold of $\bar{M}$. }
\end{enumerate}
\end{thm}

\begin{proof} Since
$$\parallel h \parallel^2= \parallel h(D^T,D^T)
\parallel^2+\parallel h(D^{\theta},D^{\theta}) \parallel^2+2\parallel h(D^T,D^{\theta})
\parallel^2,$$
we have
\begin{eqnarray}
&&\parallel h
\parallel^2=\sum_{k=1}^{m+2p}\sum_{i,j=1}^{m}g(h(e_i,e_j),\tilde{e}_k)^2+\sum_{k=1}^{m+2p}\sum_{r,s=1}^{2p}g(h(\bar{e}_r,\bar{e}_s),\tilde{e}_k)^2\nonumber\\
&&+2\sum_{k=1}^{m+2p}\sum_{r=1}^{2p}\sum_{i=1}^{m}g(h({e}_i,\bar{e}_r),\tilde{e}_k)^2\nonumber
\end{eqnarray}
where $\{ \tilde{e}_k\}$ is an orthonormal basis of $TM^{\perp}$.
Now, considering the adapted frame, we can write the above
equation as
\begin{eqnarray}
\parallel h
\parallel^2&=&\sum_{a=1}^{2p}\sum_{i,j=1}^{m}g(h(e_i,e_j),\csc \theta\,
F\bar{e}_a)^2+
\sum_{a,r,s=1}^{2p}g(h(\bar{e}_r,\bar{e}_s),\csc
\theta\,F\bar{e}_a)^2\nonumber\\
&&+2\sum_{i=1}^{m}\sum_{a,r=1}^{2p}g(h(\bar{e}_r,e_i),\csc
\theta\,F\bar{e}_a)^2.\nonumber
\end{eqnarray}
Then, from (\ref{eq:5.6}) and (\ref{eq:5.7}), we obtain
\begin{eqnarray}
&&\parallel h
\parallel^2=\sum_{a,r,s=1}^{2p}g(h(\bar{e}_r,\bar{e}_s),\csc
\theta\,F\bar{e}_a)^2+2\sum_{i=1}^{m}\sum_{a,r=1}^{2p}(\csc\,\theta)^2[(Je_i(lnf)g(\bar{e}_r,\bar{e}_a))^2\nonumber\\
&&+2Je_i(lnf)g(\bar{e}_r,\bar{e}_a)e_i(lnf)g(\bar{e}_r,T\bar{e}_a)+(e_i(lnf)g(\bar{e}_r,T\bar{e}_a))^2].\nonumber
\end{eqnarray}
Since
\begin{eqnarray}
&&\sum_{i=1}^{m}\sum_{a,r=1}^{2p}Je_i(lnf)g(\bar{e}_r,\bar{e}_a)e_i(lnf)g(\bar{e}_r,T\bar{e}_a)\nonumber\\
&&=\sum_{i=1}^{m}\sum_{a,r=1}^{2p}g(\nabla(lnf), Je_i)g(\nabla (lnf), e_i)g(\bar{e}_r,\bar{e}_a)g(\bar{e}_r,T\bar{e}_a)\nonumber\\
&&=-\sum_{a,r=1}^{2p}[\sum_{i=1}^{m}g(g(\nabla(lnf), e_i)e_i,J\nabla (lnf))]g(\bar{e}_r,\bar{e}_a)g(\bar{e}_r,T\bar{e}_a)=0,\nonumber
\end{eqnarray}
by using (\ref{eq:3.10}) we obtain
\begin{eqnarray}
\parallel h
\parallel^2&=&\sum_{a,r,s=1}^{2p}g(h(\bar{e}_r,\bar{e}_s),\csc
\theta\,F\bar{e}_a)^2+2n\| \nabla lnf\|^2[\csc^2\,\theta+\cot^2\,\theta].\nonumber
\end{eqnarray}
Thus we obtain the inequality
(\ref{eq:5.8}). If the equality sign of (\ref{eq:5.8}) holds, we
have
\begin{equation}
\sum_{a=1}^{2p}\sum_{r,s=1}^{2p}g(h(\bar{e}_r,\bar{e}_s),\csc
\theta\,F\bar{e}_a)^2=0. \label{eq:5.9}
\end{equation}
Since $M_T$ is totally geodesic in $M$, (\ref{eq:5.6}) implies that $M_T$ is
totally geodesic in $\bar{M}$.   On the other hand, (\ref{eq:5.9}) implies that $h$ vanishes on $D^{\theta}$. Since $D^{\theta}$ is a spherical distribution in $M$, it follows that $M_{\theta}$ is a totally umbilical submanifold of $\bar{M}$. Moreover, from (\ref{eq:5.6}) and (\ref{eq:5.9}) it follows that $M$ is minimal in $\bar{M}$. \end{proof}

\begin{rem} It is well known that the
semi-slant submanifolds were introduced as a generalization of
proper slant and proper CR-submanifolds. From Theorem~3.1 and
Theorem 3.2 of \cite{Sahin}, it follows that the semi-slant
submanifolds in the sense of N. Papaghiuc are not useful to
generalize the CR-warped products. But, from  Example~5.1, one can
conclude that non-trivial warped product pointwise semi-slant
submanifolds of the form $M_T\times_f M_{\theta}$ are a
generalization of CR-warped products in Kaehler manifolds.\end{rem}

\noindent Department of Mathematics\\
\noindent Inonu University\\
\noindent 44280, Malatya,Turkey.\\
\noindent E-mail:{\it bayram.sahin@inonu.edu.tr}

\end{document}